\documentclass[psamsfonts,english]{amsart}

    \usepackage[T1]{fontenc}
    \usepackage{babel}
    \usepackage{amsmath}
    \usepackage{amssymb}
    \usepackage{amsthm}
    \usepackage[mathcal]{eucal}
    \usepackage{url}
    \usepackage[all]{xypic}

    \newtheorem{teo}{Theorem}[section]
    \newtheorem*{teo*}{Theorem}
    \newtheorem{lem}[teo]{Lemma}
    \newtheorem*{cor}{Corollary}
    \theoremstyle{definition}
    \newtheorem*{defn}{Definition}
    
    \theoremstyle{remark}
    \newtheorem{rem}[teo]{Remark}

    \newcommand{\FF}{\mathbb{F}}
    \newcommand{\QQ}{\mathbb{Q}}
    \newcommand{\ZZ}{\mathbb{Z}}
    \newcommand{\NN}{\mathbb{N}}
    \newcommand{\RR}{\mathbb{R}}
    \newcommand{\CC}{\mathbb{C}}
    \newcommand{\PP}{\mathbb{P}}
    \newcommand{\C}{\mathcal{C}}

    \newcommand{\jac}[1]{\mathcal{J}_{#1}}
    \newcommand{\heltal}[1]{\mathfrak{O}_{#1}}
    \newcommand{\grp}[1]{\langle #1 \rangle}

    \DeclareMathOperator{\End}{End}
    \DeclareMathOperator{\Div}{Div}
    \DeclareMathOperator{\Mat}{Mat}

    \DeclareMathOperator{\divisor}{div}
    
    \DeclareMathOperator{\diag}{diag}
    \DeclareMathOperator{\tatemodul}{T}

\begin{document}

\title[Pairings on Jacobians]
{Pairings on Jacobians of Hyperelliptic Curves}

\author[C.R. Ravnshøj]{Christian Robenhagen Ravnshøj}

\address{Department of Mathematical Sciences \\
University of Aarhus \\
Ny Munkegade \\
Building 1530 \\
DK-8000 Aarhus C}

\email{cr@imf.au.dk}

\thanks{Research supported in part by a PhD grant from CRYPTOMAThIC}




\begin{abstract}
Consider the Jacobian of a hyperelliptic genus two curve defined over a finite field. Under certain restrictions on the
endomorphism ring of the Jacobian, we give an explicit description of all non-degenerate, bilinear, anti-symmetric and
Galois-invariant pairings on the Jacobian. From this description it follows that no such pairing can be computed more
efficiently than the Weil pairing.

To establish this result, we need an explicit description of the repre\-sentation of the Frobenius endomorphism on the
$\ell$-torsion subgroup of the Jacobian. This description is given. In particular, we show that if the characteristic
polynomial of the Frobenius endomorphism splits into linear factors modulo $\ell$, then the Frobenius is
diagonalizable.

Finally, under the restriction that the Frobenius element is an element of a certain subring of the endomorphism ring,
we prove that if the characteristic polynomial of the Frobenius endomorphism splits into linear factors modulo~$\ell$,
then the embedding degree and the total embedding degree of the Jacobian with respect to $\ell$ are the same number.
\end{abstract}

\maketitle

\section{Introduction}

In \cite{koblitz87}, Koblitz described how to use elliptic curves to construct a public key cryptosystem. To get a more
general class of groups, and possibly larger group orders, Koblitz \cite{koblitz89} then proposed using Jacobians of
hyperelliptic curves. Since Boney and Franklin \cite{boneh-franklin} proposed an identity based cryptosystem by using
the Weil pairing on an elliptic curve, pairings have been of great interest to cryptography~\cite{galbraith05}. The
next natural step then was to consider pairings on hyperelliptic curves. Galbraith \emph{et al}~\cite{galbraith07}
survey the recent research on pairings on hyperelliptic curves.

The pairing in question is usually the Weil or the Tate pairing; both pairings can be computed with Miller's
algorithm \cite{miller-algorithm}. The Tate pairing is usually preferred because  it can be computed more efficiently
than the Weil pairing, cf. \cite{galbraith01}, and it is non-degenerate over a possible smaller field extension than
the Weil pairing, cf. \cite{hess} and \cite{sil}. For elliptic curves, in most cases relevant to cryptography the
question of non-degeneracy is not an issue, cf.~\cite{balasubramanian}. This result has been generalized to any abelian
variety defined over a finite field by Rubin and Silverberg \cite[Theorem~3.1]{rubin}. The proof in \cite{rubin} uses
intrinsic properties of the Frobenius endomorphism on the abelian variety. This indicates the importance of knowing the
representation of the Frobenius endomorphism on torsion subgroups of the abelian variety. This representation has
implicitly been given by Rück \cite[proof of Lemma~4.2]{ruck}.

Cryptographically, it is essential to know the number of points on the Jacobian. Currently, the \emph{complex
multiplication method} \cite{weng03,gaudry,eisen-lauter} is the only efficient method to determine the number of points
of the Jacobian of a genus two curve defined over a large prime field \cite{gaudry}. The complex multiplication method
constructs a Jacobian with endomorphism ring isomorphic to the ring of integers $\heltal{K}$ in a \emph{quartic CM
field}~$K$, i.e. a totally imaginary, quadratic field extension of a quadratic number field. In the present paper we consider the more general situation where $\heltal{K}$ is \emph{embedded} into the endomorphism ring.

\subsection{Notation and assumptions}\label{sec:assumptions}

Consider a hyperelliptic curve $\C$ of genus two defined over a finite field~$\FF_q$ of
characteristic~$p$. We assume that the Jacobian $\jac{\C}$ of $\C$ is irreducible. Identify the $q$-power Frobenius
endo\-morphism~$\varphi$ on~$\jac{\C}$ with a root $\omega\in\CC$ of the characteristic polynomial~\mbox{$P\in\ZZ[X]$}
of~$\varphi$; cf. section~\ref{sec:RepFro}. We then assume that the ring of integers of $\QQ(\omega)$ is
\emph{embedded} into the endomorphism ring $\End(\jac{\C})$. Let $\ell\neq p$ be a prime number dividing the order
of~$\jac{\C}(\FF_q)$. Assume that $\ell$ is unramified in~$\QQ(\omega)$, and that $\ell\nmid q-1$.

\subsection{Results}

Under these assumptions, in section~\ref{sec:WeilMatrix} we give an explicit description of all non-degenerate, bilinear,
anti-symmetric, Galois-invariant pairings on the $\ell$\nobreakdash-torsion subgroup of the Jacobian of a hyperelliptic curve of
genus two, given by the following theorem.

\setcounter{section}{5} \setcounter{teo}{0}

\begin{teo}[Anti-symmetric pairings]\label{introteo:weil}
Let notation and assumptions be as above. Choose a basis $\mathcal{B}$ of $\jac{\C}[\ell]$, such that $\varphi$ is
represented either by a diagonal matrix or a matrix on the form given in theorem~\ref{teo:matrixrep} with respect to
$\mathcal{B}$. If $\jac{\C}(\FF_q)[\ell]$ is cyclic, then all non-degenerate, bilinear, anti-symmetric and
Galois-invariant pairings on $\jac{\C}[\ell]$ are given by the matrices
    $$\mathcal{E}_{a,b}=\begin{bmatrix}
            0 & a & 0 & 0 \\
            -a & 0 & 0 & 0 \\
            0 & 0 & 0 & b \\
            0 & 0 & -b & 0
            \end{bmatrix},\qquad a,b\in\FF_\ell^\times
    $$
with respect to $\mathcal{B}$.
\end{teo}

This result implies that \emph{the Weil pairing is non-degenerate on the same field extension as the Tate pairing}, and
that \emph{no non-degenerate, bilinear, anti-symmetric and Galois-invariant pairing on $\jac{\C}[\ell]$ can be computed
more effectively than the Weil pairing}. To end the description of pairings on $\jac{\C}$, in
section~\ref{sec:TateMatrix} we give an explicit description of the Tate pairing.

The proof of Theorem~\ref{introteo:weil} uses an explicit description of the representation of the Frobenius
endomorphism on the Jacobian of a hyperelliptic curve of genus two, given by the following theorem.

\setcounter{section}{4} \setcounter{teo}{1}

\begin{teo}[Matrix representation]\label{teo:matrixrep}
Let notation and assumptions be as above. Then either $\varphi$ is diagonalizable on $\jac{\C}[\ell]$, or $\varphi$ is
represented on $\jac{\C}[\ell]$ by a matrix on the form
    $$M=\begin{bmatrix}
        1 & 0 & 0 & 0 \\
        0 & q & 0 & 0 \\
        0 & 0 & 0 & -q \\
        0 & 0 & 1 & c
        \end{bmatrix}
    $$
with $c\not\equiv q+1\pmod{\ell}$ with respect to an appropriate basis of $\jac{\C}[\ell]$.
\end{teo}

Perhaps even more interestingly, we prove that if the characteristic polynomial of the Frobenius endomorphism splits into
linear factors modulo $\ell$, then the Frobenius is diagonalizable.

\setcounter{section}{4} \setcounter{teo}{6}

\begin{teo}[Diagonal representation]\label{teo:diagonalrep}
Let notation and assumptions be as above. Then $\varphi$ is diagonalizable on $\jac{\C}[\ell]$ if and only if the
characteristic polynomial of $\varphi$ splits into linear factors modulo~$\ell$.
\end{teo}

The proofs are given in section~\ref{sec:RepFro}. Theorem~\ref{teo:FrobEjDiagonal} and \ref{teo:FrobDiagonaliserbar}
also hold if $\ell\mid q-1$ and $\ell$ is uneven. The proofs are similar in this case, but due to the
\textsc{mov}-attack~\cite{mov} and the attack by Frey-Rück~\cite{frey-ruck}, the case $\ell\mid q-1$ is not of
cryptographic interest. Therefore, this case is omitted.

Finally, in section~\ref{sec:CM} we assume that the endomorphism ring of the Jacobian is \emph{isomorphic} to the ring
of integers in a quartic CM field $K$. Assuming that the Frobenius endomorphism under this isomorphism is given by an
\emph{$\eta$-integer} and that the characteristic polynomial of the Frobenius endomorphism splits into linear factors
over~$\FF_\ell$, we prove that if the discriminant of the real subfield of $K$ is not a quadratic residue modulo
$\ell$, then all $\ell$-torsion points are $\FF_{q^k}$-rational. Here, $k$ is the multiplicative order of $q$ modulo $\ell$.


\setcounter{section}{1}

\section{Hyperelliptic curves}

A hyperelliptic curve is a smooth, projective curve $\C\subseteq\PP^n$ of genus at least two with a separable, degree
two morphism $\phi:\C\to\PP^1$. Throughout this paper, let $\C$ be a hyperelliptic curve of genus two defined over a
finite field~$\FF_q$ of characteristic $p$. By the Riemann-Roch Theorem there exists a birational map
\mbox{$\psi:\C\to\PP^2$}, mapping $\C$ to a curve given by an equation of the form
    $$y^2+g(x)y=h(x),$$
where $g,h\in\FF_q[x]$ are polynomials of degree at most six \cite[chapter~1]{cassels}.

The set of principal divisors $\mathcal{P}(\C)$ on $\C$ constitutes a subgroup of the degree zero divisors $\Div_0(\C)$.
The Jacobian $\jac{\C}$ of $\C$ is defined as the quotient
    $$\jac{\C}=\Div_0(\C)/\mathcal{P}(\C).$$
The Jacobian is defined over $\FF_q$, and the points on $\jac{\C}$ defined over the extension $\FF_{q^d}$ is denoted
$\jac{\C}(\FF_{q^d})$. 

Let $\ell\neq p$ be a prime number. The $\ell^n$-torsion subgroup $\jac{\C}[\ell^n]<\jac{\C}$ of elements of order
dividing $\ell^n$ is then isomorphic to $(\ZZ/\ell^n\ZZ)^4$, i.e. $\jac{\C}[\ell^n]$ is a $\ZZ/\ell^n\ZZ$-module of
rank four; cf.~\cite[Theorem~6, p.~109]{lang59}.

The multiplicative order of $q$ modulo $\ell$ plays an important role in cryptography.

\begin{defn}[Embedding degree]
Consider a prime number $\ell\neq p$ dividing the order of $\jac{\C}(\FF_q)$. The embedding degree of $\jac{\C}(\FF_q)$
with respect to $\ell$ is the multiplicative order of $q$ modulo $\ell$, i.e. the least number $k$, such that
$q^k\equiv 1\pmod{\ell}$.
\end{defn}

Throughout this paper we consider a prime number $\ell\neq p$ dividing the order of~$\jac{\C}(\FF_q)$, and assume that
$\jac{\C}(\FF_q)$ is of embedding degree $k>1$ with respect to~$\ell$.

{\samepage
Closely related to the embedding degree, we have the \emph{total} embedding degree.

\begin{defn}[Total embedding degree]
Consider a prime number $\ell\neq p$ dividing the order of $\jac{\C}(\FF_q)$. The total embedding degree of
$\jac{\C}(\FF_q)$ with respect to $\ell$ is the least number $\kappa$, such that
$\jac{\C}[\ell]\subseteq\jac{\C}(\FF_{q^\kappa})$.
\end{defn}

\begin{rem}
If $\jac{\C}[\ell]\subseteq\jac{\C}(\FF_{q^\kappa})$, then $\ell\mid q^\kappa-1$; cf.~\cite[corollary~5.77,
p.~111]{hhec}. Hence, the total embedding degree is a multiple of the embedding degree.
\end{rem}
}

\section{The tame Tate pairing}

Let $\FF$ be an algebraic extension of $\FF_q$.
Let $x\in\jac{\C}(\FF)[\ell]$ and $y=\sum_ia_i P_i\in\jac{\C}(\FF)$ be divisors with disjoint support, and let
$\bar{y}\in\jac{\C}(\FF)/\ell\jac{\C}(\FF)$ denote the divisor class containing the divisor~$y$. Furthermore, let
$f_x\in\FF(\C)$ be a rational function on $\C$ with divisor $\divisor(f_x)=\ell x$. Set $f_x(y)=\prod_if(P_i)^{a_i}$.
Then
    $$e_\ell(x,\bar{y})=f_x(y)$$
is a well-defined pairing
$\jac{\C}(\FF)[\ell]\times\jac{\C}(\FF)/\ell\jac{\C}(\FF)\longrightarrow\FF^\times/(\FF^\times)^\ell$, the \emph{Tate
pairing}; cf.~\cite{galbraith05}.

{\samepage
\begin{teo}\label{teo:tatepairing}
If the field~$\FF$ is finite and contains the $\ell^\mathrm{th}$ roots of unity, then the Tate pairing $e_\ell$ is
bilinear and non-degenerate.
\end{teo}

\begin{proof}
Hess \cite{hess} gives a short and elementary proof of this result.
\end{proof}
}

Now let $\FF=\FF_{q^k}$. Raising to the power $\frac{q^k-1}{\ell}$ gives a well-defined element in the subgroup
$\mu_\ell<\FF_{q^k}^\times$ of the $\ell^{\mathrm{th}}$ roots of unity. This pairing
    $$\hat{e}_\ell:\jac{\C}(\FF_{q^k})[\ell]\times\jac{\C}(\FF_{q^k})/\ell\jac{\C}(\FF_{q^k})\longrightarrow\mu_\ell$$
is called the \emph{tame} Tate pairing.

\begin{cor}
The tame Tate pairing $\hat{e}_\ell$ is bilinear and non-degenerate.
\end{cor}

\section{Tate representation of the Frobenius endomorphism}\label{sec:RepFro}


Let $\ZZ_\ell$ denote the ring of $\ell$-adic integers. An endomorphism $\psi:\jac{\C}\to\jac{\C}$ induces a
$\ZZ_\ell$-linear map
    $$\psi_\ell:\tatemodul_\ell(\jac{\C})\to \tatemodul_\ell(\jac{\C})$$
on the $\ell$-adic Tate-module $\tatemodul_\ell(\jac{\C})$ of $\jac{\C}$; cf.~\cite[chapter~VII, \S1]{lang59}. The map
$\psi_\ell$ is given by $\psi$ as described in figure~\ref{fig:taterep}.
\begin{figure}[!bt]
    $$
    \xymatrix@C=40pt@R=40pt{
    \dots \ar[r]^(0.4){[\ell]} & \jac{\C}[\ell^{n+1}] \ar[r]^{[\ell]} \ar[d]^{\psi} & \jac{\C}[\ell^{n}] \ar[r]^{[\ell]} \ar[d]^{\psi} & \dots \\
    \dots \ar[r]^(0.4){[\ell]} & \jac{\C}[\ell^{n+1}] \ar[r]^{[\ell]} & \jac{\C}[\ell^{n}] \ar[r]^{[\ell]} & \dots \\
    }
    $$
\caption{Representation of an endomorphism $\psi\in\End(\jac{\C})$ on the Tate module $\tatemodul_\ell(\jac{\C})$. The
horizontal maps $[\ell]$ are the multiplication-by-$\ell$ map.}\label{fig:taterep}
\end{figure}
Hence, $\psi$ is represented on $\jac{\C}[\ell]$ by a matrix $M\in\Mat_4(\FF_\ell)$.

\begin{defn}[Diagonal representation]
An endomorphism $\psi\in\End(\jac{\C})$ is \emph{diagonalizable} or has a \emph{diagonal representation} on
$\jac{\C}[\ell]$, if $\psi$ can be represented on $\jac{\C}[\ell]$ by a diagonal matrix $M\in\Mat_4(\FF_\ell)$ with
respect to an appropriate basis of $\jac{\C}[\ell]$.
\end{defn}

Let $f\in\ZZ[X]$ be the characteristic polynomial of~$\psi$, cf.~\cite[pp.~109--110]{lang59}, and let
$\bar{f}(X)\in\FF_\ell[X]$ be the characteristic polynomial of the restriction of $\psi$ to $\jac{\C}[\ell]$. Then $f$
is a monic polynomial of degree four, and by \cite[Theorem~3, p.~186]{lang59},
    \begin{equation*}\label{eq:KarPolKongruens}
    f(X)\equiv \bar{f}(X)\pmod{\ell}.
    \end{equation*}


Since $\C$ is defined over $\FF_q$, the mapping $(x,y)\mapsto (x^q,y^q)$ is a morphism on~$\C$. This morphism induces
the $q$-power Frobenius endo\-morphism $\varphi$ on the Jacobian~$\jac{\C}$. Let $P$ be the characteristic polynomial
of $\varphi$. Consider an algebraic integer $\omega\in\CC$ with $P(\omega)=0$ in~$\CC$. By the homomorphism
$\ZZ[\omega]\to\End(\jac{\C})$ given by $\omega\mapsto\varphi$ we will identify $\varphi$ with $\omega$.

Since $\End(\jac{\C})$ is a finitely generated, torsion free $\ZZ$-module \cite[Theorem~1]{milne-waterhouse}, we may
define \mbox{$\End_\QQ(\jac{\C})=\End(\jac{\C})\otimes\QQ$}. Notice that $\QQ(\omega)\subseteq\End_\QQ(\jac{\C})$.
Throughout this paper we assume that $\ell$ is unramified in~$\QQ(\omega)$.

\begin{rem}
It is well-known that $\ell$ is unramified in~$\QQ(\omega)$ if and only if $\ell$ divides the discriminant of the
field extension $\QQ(\omega)/\QQ$; see e.g. \cite[Theorem~2.6, p.~199]{neukirch}. Hence, almost any prime number $\ell$
is unramified in~$\QQ(\omega)$. In particular, if $\ell$ is large, then $\ell$ is unramified in~$\QQ(\omega)$.
\end{rem}

We prove the following theorem.

\begin{teo}[Matrix representation]\label{teo:FrobEjDiagonal}
Let $\C$ be a hyperelliptic curve of genus two defined over a finite field~$\FF_q$ of characteristic~$p$ with
irreducible Jacobian. Identify the $q$-power Frobenius endomorphism $\varphi$ on~$\jac{\C}$ with a root $\omega\in\CC$
of the characteristic polynomial $P\in\ZZ[X]$ of $\varphi$. Assume that the ring of integers of $\QQ(\omega)$ under
this identification is embedded in $\End(\jac{\C})$. Consider a prime number $\ell\neq p$ dividing the order of
$\jac{\C}(\FF_q)$. Assume that $\ell$ is unramified in~$\QQ(\omega)$, and that $\ell\nmid q-1$. If $\varphi$ is not
diagonalizable on $\jac{\C}[\ell]$, then $\varphi$ is represented on $\jac{\C}[\ell]$ by a matrix on the form
    \begin{equation}\label{eq:M}
    M=\begin{bmatrix}
        1 & 0 & 0 & 0 \\
        0 & q & 0 & 0 \\
        0 & 0 & 0 & -q \\
        0 & 0 & 1 & c
        \end{bmatrix}
    \end{equation}
with $c\not\equiv q+1\pmod{\ell}$ with respect to an appropriate basis of $\jac{\C}[\ell]$.
\end{teo}

The proof of theorem~\ref{teo:FrobEjDiagonal} uses a number of lemmas. At first we notice that if a power of an
endomorphism is trivial on the $\ell$-torsion subgroup of $\jac{\C}$, then so is also the endomorphism.

\begin{lem}\label{lem:unramified}
Let notation and assumptions be as in theorem~\ref{teo:FrobEjDiagonal}. Consider an endomorphism
$\alpha\in\QQ(\omega)$. If $\ker[\ell]\subseteq\ker(\alpha^n)$ for some number $n\in\NN$, then
$\ker[\ell]\subseteq\ker(\alpha)$.
\end{lem}

\begin{proof}
Since $\ker[\ell]\subseteq\ker(\alpha^n)$, it follows that
$\alpha^n=\ell\beta$ for some endomorphism $\beta\in\End(\jac{\C})$; see e.g. \cite[Remark~7.12, p.~37]{milne:AV}.
Notice that $\beta=\frac{\alpha^n}{\ell}\in\QQ(\omega)$. Let $f\in\ZZ[X]$ be the characteristic polynomial of~$\beta$.
Since $f(\beta)=0$ and $f$ is monic, $\beta$ is an algebraic integer. So $\beta\in\heltal{\QQ(\omega)}$, whence
$\alpha^n\in\ell\heltal{\QQ(\omega)}$. Since $\ell$ is unramified in~$\QQ(\omega)$ by assumption, it follows that
$\alpha\in\ell\heltal{\QQ(\omega)}$, i.e. $\ker[\ell]\subseteq\ker(\alpha)$.
\end{proof}

We will examine the representation of $\varphi$ on $\jac{\C}[\ell]$. A first, basic observation is given by the
following lemma.

\begin{lem}\label{lem:rank}
Let notation and assumptions be as in theorem~\ref{teo:FrobEjDiagonal}. Then either $\jac{\C}(\FF_{q^k})[\ell]$ is of
dimension two as a $\FF_\ell$-vectorspace, or all $\ell$-torsion points of $\jac{\C}$ are $\FF_{q^k}$-rational.
\end{lem}

\begin{proof}
By the non-degeneracy of the Tate pairing on $\jac{\C}(\FF_{q^k})[\ell]$, the dimension over $\FF_\ell$ is at least
two. If $\jac{\C}(\FF_{q^k})[\ell]$ is of dimension at least three over $\FF_\ell$, then the restriction of the
$q^k$-power Frobenius endomorphism $\varphi^k$ to $\jac{\C}(\FF_{q^k})[\ell]$ is represented by a matrix on the form
    $$
    M=\begin{bmatrix}
    1 & 0 & 0 & m_1 \\
    0 & 1 & 0 & m_2 \\
    0 & 0 & 1 & m_3 \\
    0 & 0 & 0 & m_4
    \end{bmatrix}.
    $$
Notice that $m_4=\det M\equiv\deg(\varphi^k)=q^{2k}\equiv 1\pmod{\ell}$. Thus, the characteristic polynomial of
$\varphi^k$ satisfies $P(X)\equiv (X-1)^4\pmod{\ell}$, i.e. $\ker[\ell]\subseteq\ker(\varphi^k-1)^4$. By
Lemma~\ref{lem:unramified} it follows that $\ker[\ell]\subseteq\ker(\varphi^k-1)$. But then
$\jac{\C}[\ell]\subseteq\jac{\C}(\FF_{q^k})$, i.e. all $\ell$-torsion points of $\jac{\C}$ are $\FF_{q^k}$-rational.
\end{proof}

By \cite[proof of Theorem~3.1]{rubin} we know that $\jac{\C}[\ell]$ as a vector space over $\FF_\ell$ is isomorphic to
a direct sum of $\varphi$-invariant subspaces. From this we get a partial description of the representation of
$\varphi$ on $\jac{\C}[\ell]$.

{\samepage
\begin{lem}\label{lem:M}
Let notation and assumptions be as in theorem~\ref{teo:FrobEjDiagonal}. We may choose a basis $(x_1,x_2,x_3,x_4)$ of
$\jac{\C}[\ell]$, where $\varphi(x_1)=x_1$, $\varphi(x_2)=qx_2$ and $\varphi(x_3)\in\grp{x_3,x_4}$. If
$\varphi(x_3)\notin\grp{x_3}$, then $\varphi$ can be represented on $\jac{\C}[\ell]$ by a matrix on the form
    \begin{equation*}
    M=\begin{bmatrix}
        1 & 0 & 0 & 0 \\
        0 & q & 0 & 0 \\
        0 & 0 & 0 & -q \\
        0 & 0 & 1 & c
        \end{bmatrix}.
    \end{equation*}
If $c\equiv q+1\pmod{\ell}$, then $\varphi$ is diagonalizable.
\end{lem} }

\begin{proof}
Let $\bar{P}\in\FF_\ell[X]$ be the characteristic polynomial of the restriction of $\varphi$ to~$\jac{\C}[\ell]$. Since
$\ell\mid|\jac{\C}(\FF_q)|$, $1$ is a root of $\bar{P}$. Assume that $1$ is an root of $\bar{P}$ with multiplicity~$d$.
Since the roots of $\bar{P}$ occur in pairs $(\alpha,q/\alpha)$, also $q$ is a root of $\bar{P}$ with multiplicity~$d$.
Hence, we may write
    $$\bar{P}(X)=(X-1)^d(X-q)^d\bar{Q}(X),$$
where $\bar{Q}\in\FF_\ell[X]$ is a polynomial of degree $4-2d$, and $\bar{Q}(1)\cdot\bar{Q}(q)\not\equiv 0\pmod{\ell}$.
Let $U=\ker(\varphi-1)^d$, $V=\ker(\varphi-q)^d$ and $W=\ker(\bar{Q}(\varphi))$. Then $U$, $V$ and $W$ are
$\varphi$-invariant subspaces of the $\FF_\ell$-vectorspace $\jac{\C}[\ell]$,
$\dim_{\FF_\ell}(U)=\dim_{\FF_\ell}(V)=d$, and $\jac{\C}[\ell]\simeq U\oplus V\oplus W$.

If $d=1$, then choose $x_i\in\jac{\C}[\ell]$, such that $U=\grp{x_1}$, $V=\grp{x_2}$ and $W=\grp{x_3,x_4}$. Then
$(x_1,x_2,x_3,x_4)$ establishes the first part of the lemma. Hence, we may assume that $d=2$. Now choose $x_1\in U$,
such that $\varphi(x_1)=x_1$, and expand this to a basis $(x_1,x_2)$ of $U$. Similarly, choose a basis $(x_3,x_4)$ of
$V$ with $\varphi(x_3)=qx_3$. With respect to the basis $(x_1,x_2,x_3,x_4)$, $\varphi$ is then represented by a matrix
on the form
    $$
    M=\begin{bmatrix}
      1 & \alpha & 0 & 0 \\
      0 & 1 & 0 & 0 \\
      0 & 0 & q & \beta \\
      0 & 0 & 0 & q
      \end{bmatrix}.
    $$
Notice that
    $$
    M^k=\begin{bmatrix}
      1 & k\alpha & 0 & 0 \\
      0 & 1 & 0 & 0 \\
      0 & 0 & 1 & k q^{k-1}\beta \\
      0 & 0 & 0 & 1
      \end{bmatrix}.
    $$
Hence, the restriction of $\varphi^k$ to $\jac{\C}[\ell]$ has the characteristic polynomial $(X-1)^4$, i.e.
$\jac{\C}[\ell]\subseteq\jac{\C}(\FF_{q^k})$. But then $M^k=I$, whence $\alpha\equiv\beta\equiv 0\pmod{\ell}$. So if
$d=2$, then the first part of the lemma is established by $(x_1,x_3,x_2,x_4)$. Thus, the first part of the lemma is
proved.

Now choose a basis $(x_1,x_2,x_3,x_4)$ of $\jac{\C}[\ell]$ according to the first part of the lemma. Assume that
$\varphi(x_3)\notin\grp{x_3}$. Then the set $(x_1,x_2,x_3,\varphi(x_3))$ is a basis of~$\jac{\C}[\ell]$. With respect
to this basis, $\varphi$ is represented by a matrix on the given form. If $c\equiv q+1\pmod{\ell}$, then $\varphi$ is
diagonalizable.
\end{proof}

\begin{rem}\label{rem:M}
Notice that if $\bar{P}(X)=(X-1)^2(X-q)^2$, then $\varphi$ is represented by the dia\-gonal matrix $\diag(1,1,q,q)$
with respect to an appropriate basis of $\jac{\C}[\ell]$, $\jac{\C}(\FF_q)[\ell]$ is bi-cyclic and
$\jac{\C}[\ell]\subseteq\jac{\C}(\FF_{q^k})$.
\end{rem}

With lemma~\ref{lem:M} we can finally prove theorem~\ref{teo:FrobEjDiagonal}.

\begin{proof}[Proof of theorem~\ref{teo:FrobEjDiagonal}]
If $\varphi(x_3)\in\grp{x_3}$, then $\varphi$ is represented by a matrix on the form
    $$M=\begin{bmatrix}
        1 & 0 & 0 & 0 \\
        0 & q & 0 & 0 \\
        0 & 0 & \alpha & \beta \\
        0 & 0 & 0 & q\alpha^{-1}
        \end{bmatrix}
    $$
with respect to $(x_1,x_2,x_3,x_4)$. If $\alpha^2\not\equiv q\pmod{\ell}$, then $M$ is diagonalizable, i.e.~$\varphi$
can be represented by a diagonal matrix on $\jac{\C}[\ell]$. So assume that $\alpha^2\equiv q\pmod{\ell}$. Then
    $$M^{2k}=\begin{bmatrix}
        1 & 0 & 0 & 0 \\
        0 & 1 & 0 & 0 \\
        0 & 0 & 1 & 2k\alpha^{-1}\beta \\
        0 & 0 & 0 & 1
        \end{bmatrix},
    $$
i.e. the restriction of $\varphi^{2k}$ to $\jac{\C}[\ell]$ has the characteristic polynomial $(X-1)^4$. But then
$\jac{\C}[\ell]\subseteq\jac{\C}(\FF_{q^{2k}})$ by Lemma~\ref{lem:unramified}, i.e. $M^{2k}=I$. So $\beta\equiv
0\pmod{\ell}$, and $\varphi$ is diagonalizable.

Thus, if $\varphi$ is not diagonalizable on $\jac{\C}[\ell]$, then $\varphi(x_3)\notin\grp{x_3}$, whence $\varphi$ is
represented on $\jac{\C}[\ell]$ by a matrix on the form~\eqref{eq:M} with respect to an appropriate basis of
$\jac{\C}[\ell]$.
\end{proof}

Since the roots of the characteristic polynomial $P$ of the Frobenius $\varphi$ are all of absolute value $\sqrt{q}$,
we can determine whether the Frobenius is diagonalizable on $\jac{\C}[\ell]$ directly from the roots of $P$ modulo
$\ell$. From this it follows that if $P$  splits into linear factors modulo $\ell$, then the Frobenius is
diagonalizable.

\begin{teo}[Diagonal representation]\label{teo:FrobDiagonaliserbar}
Let notation and assumptions be as in theorem~\ref{teo:FrobEjDiagonal}. Then $\varphi$ is diagonalizable on
$\jac{\C}[\ell]$ if and only if the characteristic polynomial of $\varphi$ splits into linear factors modulo~$\ell$.
\end{teo}

\begin{proof}
The ``only if'' part is trivial. We prove the ``if'' part.

Let $\bar{P}\in\FF_\ell[X]$ be the characteristic polynomial of the restriction of $\varphi$ to $\jac{\C}[\ell]$.
Assume at first that $\jac{\C}(\FF_q)[\ell]$ is cyclic. If $\bar{P}(X)=(X-1)^2(X-q)^2$, then $\jac{\C}[\ell]$ is
bi-cyclic by Remark~\ref{rem:M}. So $\bar{P}(X)\neq (X-1)^2(X-q)^2$. If $\bar{P}$ has only simple roots, then $\varphi$
is diagonalizable. Hence, we may assume that $\bar{P}$ has a double root $\bar{\alpha}\in\FF_\ell$. The roots of
$\bar{P}$ occur in pairs $(\bar{\alpha},q/\bar{\alpha})$. Thus, if $\bar{\alpha}\in\{1,q\}$, then
$\bar{P}(X)=(X-1)^2(X-q)^2$. So $\bar{\alpha}\notin\{1,q\}$, and it follows that $\varphi$ can be represented on
$\jac{\C}[\ell]$ by a matrix on the form
    $$
    M=\begin{bmatrix}
    1 & 0 & 0 & 0 \\
    0 & q & 0 & 0 \\
    0 & 0 & \alpha & \beta \\
    0 & 0 & 0 & \alpha
    \end{bmatrix},
    $$
where $\alpha\equiv\bar{\alpha}\pmod{\ell}$. Let $\alpha^\kappa\equiv 1\pmod{\ell}$. Then
    $$
    M^\kappa=\begin{bmatrix}
    1 & 0 & 0 & 0 \\
    0 & 1 & 0 & 0 \\
    0 & 0 & 1 & \kappa\alpha^{\kappa-1}\beta \\
    0 & 0 & 0 & 1
    \end{bmatrix},
    $$
i.e. the restriction of $\varphi^\kappa$ to $\jac{\C}[\ell]$ has the characteristic polynomial $(X-1)^4$. But then
$\jac{\C}[\ell]\subseteq\jac{\C}(\FF_{q^\kappa})$ by Lemma~\ref{lem:unramified}, i.e. $M^\kappa=I$. So $\beta\equiv
0\pmod{\ell}$, and $\varphi$ is diagonalizable.

Then assume that $\jac{\C}(\FF_q)[\ell]$ is bi-cyclic. Then $\jac{\C}[\ell]\subseteq\jac{\C}(\FF_q)$ by
Lemma~\ref{lem:rank}, and it follows that $\varphi$ can be represented on $\jac{\C}[\ell]$ by a matrix on the form
    $$
    M=\begin{bmatrix}
    1 & 0 & 0 & 0 \\
    0 & 1 & 0 & 0 \\
    0 & 0 & q & \alpha \\
    0 & 0 & 0 & q
    \end{bmatrix}.
    $$
As above, it follows that $\alpha\equiv 0\pmod{\ell}$, whence $\varphi$ is diagonalizable.
\end{proof}

\begin{rem}
Assume that $P$ splits into linear factors modulo $\ell$. If $\jac{\C}(\FF_q)[\ell]$ is cyclic, then $\varphi$ is
diagonalizable on $\jac{\C}[\ell]$, and the the total embedding degree $\kappa$ of $\jac{\C}(\FF_q)$ with respect
to~$\ell$ is given by the multiplicative order of a root $\alpha\in\FF_\ell$ of $\bar{P}$. If $\jac{\C}[\ell]$ is not
cyclic, then $\jac{\C}[\ell]\subseteq\jac{\C}(\FF_{q^k})$ by Lemma~\ref{lem:rank}, i.e. $\kappa=k$. Hence, $\kappa$ is
easy to determine.
\end{rem}

\section{Anti-symmetric pairings on the Jacobian}\label{sec:WeilMatrix}

On $\jac{\C}[\ell]$, a non-degenerate, bilinear, anti-symmetric and Galois-invariant pairing
    $$\varepsilon:\jac{\C}[\ell]\times\jac{\C}[\ell]\to\mu_\ell<\FF_{q^k}^\times$$
exists, e.g. the Weil pairing. Since $\varepsilon$ is bilinear, it is given by
    $$\varepsilon(x,y)=x^T\mathcal{E}y$$
for some matrix $\mathcal{E}\in\Mat_4(\FF_\ell)$ with respect to a basis $(x_1,x_2,x_3,x_4)$ of $\jac{\C}[\ell]$. Since
$\varepsilon$ is Galois--invariant,
    $$\forall x,y\in\jac{\C}[\ell]: \varepsilon(x,y)^q=\varepsilon(\varphi(x),\varphi(y)).$$
This is equivalent to
    $$\forall x,y\in\jac{\C}[\ell]: q(x^T\mathcal{E}y)=(Mx)^T\mathcal{E}(My),$$
where $M$ is the representation of $\varphi$ on $\jac{\C}[\ell]$ with respect to $(x_1,x_2,x_3,x_4)$. Since
$(Mx)^T\mathcal{E}(My)=x^TM^T\mathcal{E}My$, from the Galois-invariance of $\varepsilon$ it follows that
    $$\forall x,y\in\jac{\C}[\ell]:x^Tq\mathcal{E}y=x^TM^T\mathcal{E}My,$$
or equivalently, that $q\mathcal{E}=M^T\mathcal{E}M$.

Now let $\zeta$ be a primitive $\ell^\mathrm{th}$ root of unity.  Let
    $$
    \varepsilon(x_1,x_2) = \zeta^{a_1}, \quad
    \varepsilon(x_1,x_3) = \zeta^{a_2}, \quad
    \varepsilon(x_2,x_3) = \zeta^{a_4} \quad
    \textrm{and} \quad
    \varepsilon(x_3,x_4) = \zeta^{a_6}.
    $$
Assume at first that $\varphi$ is not diagonalizable on~$\jac{\C}[\ell]$. By Galois-invariance and anti{\-}-sym\-me\-try
we then see that
    $$\mathcal{E}=\begin{bmatrix}
        0 & a_1 & a_2 & qa_2 \\
        -a_1 & 0 & a_4 & a_4 \\
        -a_2 & -a_4 & 0 & a_6 \\
        -qa_2 & -a_4 & -a_6 & 0
        \end{bmatrix}.$$
Since $M^T\mathcal{E}M=q\mathcal{E}$, it follows that
    $$a_2q(c-(1+q))\equiv a_4q(c-(1+q))\equiv 0\pmod{\ell}.$$
Thus, $a_2\equiv a_4\equiv 0\pmod{\ell}$, cf. Theorem~\ref{teo:FrobEjDiagonal}. So
    \begin{equation}\label{eq:weil}
        \mathcal{E}=\begin{bmatrix}
        0 & a_1 & 0 & 0 \\
        -a_1 & 0 & 0 & 0 \\
        0 & 0 & 0 & a_6 \\
        0 & 0 & -a_6 & 0
        \end{bmatrix}.
    \end{equation}
Since $\varepsilon$ is non-degenerate, $a_1^2a_6^2=\det \mathcal{E}\not\equiv 0\pmod{\ell}$.

Now assume that $\varphi$ is represented by a diagonal matrix $\diag(1,q,\alpha,q/\alpha)$ with respect to an
appropriate basis $(x_1,x_2,x_3,x_4)$ of $\jac{\C}[\ell]$. Let $\varepsilon(x_1,x_4)=\zeta^{a_3}$ and
$\varepsilon(x_1,x_4)=\zeta^{a_5}$. Then it follows from $M^T\mathcal{E}M=q\mathcal{E}$ that
    $$a_2(\alpha-q)\equiv a_3(\alpha-1)\equiv a_4(\alpha-1)\equiv a_5(\alpha-q)\equiv 0\pmod{\ell}.$$
If $\alpha\equiv 1,q\pmod{\ell}$, then $\jac{\C}(\FF_q)$ is bi-cyclic. Hence the following theorem holds.

\begin{teo}[Anti-symmetric pairings]\label{teo:anti-symmetric-pairings}
Let $\C$ be a hyperelliptic curve of genus two defined over a finite field~$\FF_q$ of characteristic~$p$ with
irreducible Jacobian. Identify the $q$-power Frobenius endomorphism $\varphi$ on~$\jac{\C}$ with a root $\omega\in\CC$
of the characteristic polynomial $P\in\ZZ[X]$ of $\varphi$. Assume that the ring of integers of $\QQ(\omega)$ under
this identification is embedded in $\End(\jac{\C})$. Choose a basis $\mathcal{B}$ of $\jac{\C}[\ell]$, such that
$\varphi$ is represented either by a diagonal matrix or a matrix on the form given in theorem~\ref{teo:FrobEjDiagonal}
with respect to $\mathcal{B}$. Consider a prime number $\ell\neq p$ dividing the order of $\jac{\C}(\FF_q)$. Assume
that $\ell$ is unramified in~$\QQ(\omega)$, and that $\ell\nmid q-1$. If $\jac{\C}(\FF_q)[\ell]$ is cyclic, then all
non-degenerate, bilinear, anti-symmetric and Galois-invariant pairings on $\jac{\C}[\ell]$ are given by the matrices
    $$\mathcal{E}_{a,b}=\begin{bmatrix}
            0 & a & 0 & 0 \\
            -a & 0 & 0 & 0 \\
            0 & 0 & 0 & b \\
            0 & 0 & -b & 0
            \end{bmatrix},\qquad a,b\in\FF_\ell^\times
    $$
with respect to $\mathcal{B}$.
\end{teo}

{\samepage
\begin{cor}
Under the assumptions of theorem~\ref{teo:anti-symmetric-pairings},
 \begin{enumerate}
 \item the Weil-pairing is non-degenerate on $\jac{\C}(\FF_{q^k})[\ell]$, and
 \item no non-degenerate, bilinear, anti-sym\-me\-tric and Galois-invariant pairing on $\jac{\C}[\ell]\times\jac{\C}[\ell]$
       can be computed more than eight times as effective as the Weil-pairing.
 \end{enumerate}
\end{cor}
}

\begin{proof}
By a precomputation, a basis $(x_1,x_2,x_3,x_4)$ of $\jac{\C}[\ell]$ can be found, such that the Weil-pairing is given
by the matrix $\mathcal{E}_{1,1}$; cf. the notation of theorem~\ref{teo:anti-symmetric-pairings}. To compute the
Weil-pairing of $A,B\in\jac{\C}[\ell]$, we only need to find the coordinates of $A$ and $B$ in this basis. Now assume that a
non-degenerate, bilinear, anti-symmetric and Galois-invariant pairing $\varepsilon$ on
$\jac{\C}[\ell]\times\jac{\C}[\ell]$ exists, such that $\varepsilon$ can be computed more than eight times as effectively as
the Weil-pairing. By a precomputation we can find the matrix representation $\mathcal{E}_{a,b}$ of $\varepsilon$. Write
$A=\sum_i\alpha_ix_i$. Then
    \begin{align*}
    \alpha_1 &= -a^{-1}\varepsilon(x_2,A), &
    \alpha_2 &= a^{-1}\varepsilon(x_1,A), \\
    \alpha_3 &= -b^{-1}\varepsilon(x_4,A), &
    \alpha_4 &= b^{-1}\varepsilon(x_3,A).
    \end{align*}
Similarly we find the coordinates of $B$. Hence, the Weil-pairing of $A$ and $B$ can be computed by at most eight
pairing computations with $\varepsilon$, a contradiction.
\end{proof}

\section{Matrix representation of the tame Tate pairing}\label{sec:TateMatrix}

The tame Tate pairing induces a pairing $\tau_\ell:\jac{\C}[\ell]\times\jac{\C}[\ell]\to\mu_\ell$ by
    $$\tau_\ell(x,y)=\hat{e}_\ell(x,\bar{y}).$$
In this section we will examine the matrix representation of this pairing.

Let $x,y\in\jac{\C}[\ell]=\jac{\C}(\FF_{q^\kappa})[\ell]$ be divisors with disjoint support, and choose functions
$f_x,f_y\in\FF_{q^\kappa}(\C)$ with $\divisor(f_x)=\ell x$ and $\divisor(f_y)=\ell y$. The Weil pairing
$e_\ell:\jac{\C}[\ell]\times\jac{\C}[\ell]\to\mu_\ell$ is then defined by
    $$e_\ell(x,y)=\frac{f_x(y)}{f_y(x)}$$
Notice that
    \begin{equation}\label{eq:weil<->tate}
    e_\ell(x,y)=\frac{\tau_\ell(x,y)}{\tau_\ell(y,x)}
    \end{equation}
Now choose an appropriate basis $(x_1,x_2,x_3,x_4)$ of $\jac{\C}[\ell]$, such that the Weil pairing is represented by
the matrix
    $$
    \mathcal{W}=\begin{bmatrix}
        0 & 1 & 0 & 0 \\
        -1 & 0 & 0 & 0 \\
        0 & 0 & 0 & 1 \\
        0 & 0 & -1 & 0
        \end{bmatrix}
    $$
with respect to this basis. Notice that $x_1\in\jac{\C}(\FF_q)$, so $\tau_\ell(x_1,x_1)=1$.

By \eqref{eq:weil<->tate} it follows that $\tau_\ell$ is represented by a matrix on the form
    $$
    \mathcal{T}=\begin{bmatrix}
        0 & a_1 & a_2 & a_3 \\
        a_1-1 & d_2 & a_4 & a_5 \\
        a_2 & a_4 & d_3 & a_6 \\
        a_3 & a_5 & a_6-1 & d_4
        \end{bmatrix}
    $$
with respect to the basis $(x_1,x_2,x_3,x_4)$.  Since $\tau_\ell$ is Galois-invariant, it follows that
$M^T\mathcal{T}M=q\mathcal{T}$, where $M$ is the representation of $\varphi$ on $\jac{\C}[\ell]$ with respect to
$(x_1,x_2,x_3,x_4)$.

Assume at first that the Frobenius $\varphi$ is not diagonalizable on $\jac{\C}[\ell]$. Then $\varphi$ is represented
by a matrix $M$ on the form given in theorem~\ref{teo:FrobEjDiagonal}, and it follows from
$M^T\mathcal{T}M=q\mathcal{T}$, that
    \begin{equation*}
    \mathcal{T}=\begin{bmatrix}
        0 & a_1 & 0 & 0 \\
        a_1-1 & 0 & 0 & 0 \\
        0 & 0 & d_3 & a_6 \\
        0 & 0 & a_6-1 & qd_3
        \end{bmatrix},
    \end{equation*}
where $2a_6\equiv d_3c+1\pmod{\ell}$.

Now assume that $\varphi$ is represented by a diagonal matrix $\diag(1,q,\alpha,q/\alpha)$ with respect to an
appropriate basis $(x_1,x_2,x_3,x_4)$ of $\jac{\C}[\ell]$. It then follows that
    $$a_i(\alpha-q)\equiv a_j(\alpha-1)\equiv d_2(q-1)\equiv d_j(\alpha^2-q)\equiv 0\pmod{\ell}$$
for $i\in\{2,5\}$ and $j\in\{3,4\}$. Hence the following theorem is established.

\begin{teo}\label{teo:tate}
Let $\C$ be a hyperelliptic curve of genus two defined over a finite field~$\FF_q$ of characteristic~$p$ with
irreducible Jacobian. Identify the $q$-power Frobenius endomorphism $\varphi$ on~$\jac{\C}$ with a root $\omega\in\CC$
of the characteristic polynomial $P\in\ZZ[X]$ of $\varphi$. Assume that the ring of integers of $\QQ(\omega)$ under
this identification is embedded in $\End(\jac{\C})$. Consider a prime number $\ell\neq p$ dividing the order of
$\jac{\C}(\FF_q)$. Assume that $\ell$ is unramified in~$\QQ(\omega)$, and that $\jac{\C}(\FF_q)$ is of embedding
degree~$k>1$ with respect to $\ell$. If $\jac{\C}(\FF_q)[\ell]$ is cyclic, then the tame Tate pairing is represented on
$\jac{\C}[\ell]\times\jac{\C}[\ell]$ by a matrix on the form
        $$
        \mathcal{T}=\begin{bmatrix}
        0 & a_1 & 0 & 0 \\
        a_1-1 & 0 & 0 & 0 \\
        0 & 0 & d_3 & a_6 \\
        0 & 0 & a_6-1 & d_4
            \end{bmatrix}
        $$
with respect to an appropriate basis of $\jac{\C}[\ell]$. Furthermore, the following holds.
\begin{enumerate}
    \item If the $q$-power Frobenius endomorphism is not diagonalizable on $\jac{\C}[\ell]$, then $d_4\equiv
    qd_3\pmod{\ell}$ and $2a_6\equiv d_3c+1\pmod{\ell}$.
    \item If the $q$-power Frobenius endomorphism is diagonalizable on $\jac{\C}[\ell]$, and
    $\jac{\C}[\ell]\not\subseteq\jac{\C}(\FF_{q^{2k}})$, then $d_3\equiv d_4\equiv 0\pmod{\ell}$.
    \item Assume $\jac{\C}(\FF_{q^k})[\ell]$ is bi-cyclic.
    \begin{enumerate}
    \item If $\ell^3\nmid|\jac{\C}(\FF_{q^k})|$, then $a_1\not\equiv 0,1\pmod{\ell}$.
    \item If $\ell^3\mid|\jac{\C}(\FF_{q^k})|$ and $\ell^2\nmid|\jac{\C}(\FF_q)|$, then $a_1\equiv 0\pmod{\ell}$.
    \end{enumerate}
\end{enumerate}
\end{teo}

\begin{proof}
Write $\jac{\C}(\FF_{q^k})[\ell]=\grp{x_1}\oplus\grp{x_2}$, where $\jac{\C}(\FF_q)[\ell]=\grp{x_1}$. If
$\ell^2\nmid|\jac{\C}(\FF_q)|$ and $\ell^3\nmid|\jac{\C}(\FF_{q^k})|$, then
$\jac{\C}(\FF_{q^k})/\ell\jac{\C}(\FF_{q^k})\simeq\jac{\C}(\FF_{q^k})[\ell]$. By Theorem~\ref{teo:tatepairing} it then
follows that $a_1\not\equiv 0,1\pmod{\ell}$. On the other hand, if $\ell^3\mid|\jac{\C}(\FF_{q^k})|$, then
$x_2\in\ell\jac{\C}(\FF_{q^k})$, i.e. $a_1\equiv 0\pmod{\ell}$.
\end{proof}

{\samepage
\begin{cor}
Assume $\ell^3\nmid|\jac{\C}(\FF_{q^k})|$. If the Frobenius is not diagonalizable on $\jac{\C}[\ell]$, then either
    \begin{enumerate}
    \item a point $x\in\jac{\C}[\ell]$ with $\tau_\ell(x,x)\neq 1$ exists, or
    \item $\tau_\ell$ is non-degenerate on $\jac{\C}[\ell]$.
    \end{enumerate}
\end{cor}
}

\begin{proof}
Choose an appropriate basis $(x_1,x_2,x_3,x_4)$ of $\jac{\C}[\ell]$, such that the Frobenius is represented by a matrix
$M$ on the form given in theorem~\ref{teo:FrobEjDiagonal}, and $\tau_\ell$ is represented by a matrix~$\mathcal{T}$ on
the form given in Theorem~\ref{teo:tate} with respect to this basis. Since $M^T\mathcal{T}M=q\mathcal{T}$, it follows
that $d_3c\equiv 2a_6-1\pmod{\ell}$. Hence, if $2a_6\not\equiv 1\pmod{\ell}$, then $d_3\not\equiv 0\pmod{\ell}$, and
$\tau$ is a self-pairing on $\jac{\C}[\ell]$. If $2a_6\equiv 1\pmod{\ell}$ and $d_3\equiv 0\pmod{\ell}$, then
$\tau_\ell$ is non-degenerate on $\jac{\C}[\ell]$.
\end{proof}

\section{Complex multiplication curves}\label{sec:CM}

In this section we assume that the endomorphism ring of the Jacobian is isomorphic to the ring of integers in a
\emph{quartic CM field} $K$, i.e. a totally imaginary, quadratic field extension of a quadratic number field. Assuming
that the Frobenius endomorphism under this isomorphism is given by an \emph{$\eta$-integer} and that the characteristic
polynomial of the Frobenius endomorphism splits into linear factors over~$\FF_\ell$, we prove that if the discriminant
of the real subfield of $K$ is not a quadratic residue modulo $\ell$, then all $\ell$-torsion points are
$\FF_{q^k}$-rational.

\subsection{Complex multiplication}

An elliptic curve $E$ with $\ZZ\neq\End(E)$ is said to have \emph{complex multiplication}. Let $K$ be an ima\-ginary,
quadratic number field with ring of integers $\heltal{K}$. $K$ is a \emph{CM field}, and if
\mbox{$\End(E)\simeq\heltal{K}$}, then $E$ is said to have \emph{CM by $\heltal{K}$}. More generally a CM field is
defined as follows.

\begin{defn}[CM field]
A number field~$K$ is a CM field, if $K$ is a totally imaginary, quadratic extension of a totally real number
field~$K_0$.
\end{defn}

We only consider \emph{quartic} CM field, i.e. CM fields of degree $[K:\QQ]=4$.

\begin{rem}\label{rem:quarticCM}
Consider a quartic CM field~$K$. Let $K_0=K\cap\RR$ be the real subfield of $K$. Then $K_0$ is a real, quadratic number
field, $K_0=\QQ(\sqrt{D})$. By a basic result on quadratic number fields, the ring of integers of $K_0$ is given by
$\heltal{K_0}=\ZZ+\xi\ZZ$, where
    $$
    \xi=\begin{cases}
    \sqrt{D}, & \textrm{if $D\not\equiv 1\pmod{4}$,} \\
    \frac{1+\sqrt{D}}{2}, & \textrm{if $D\equiv 1\pmod{4}$.}
    \end{cases}
    $$
Since $K$ is a totally imaginary, quadratic extension of $K_0$, a number $\eta\in K$ exists, such that $K=K_0(\eta)$,
$\eta^2\in K_0$. The number $\eta$ is totally imaginary, and we may assume that $\eta=i\eta_0$, $\eta_0\in\RR$.
Furthermore we may assume that $-\eta^2\in\heltal{K_0}$; so $\eta=i\sqrt{a+b\xi}$, where $a,b\in\ZZ$.
\end{rem}

Let $\C$ be a hyperelliptic curve of genus two. Then $\C$ is said to have CM by~$\heltal{K}$, if
$\End(\jac{\C})\simeq\heltal{K}$. The structure of $K$ determines whether $\jac{\C}$ is irreducible. More precisely,
the following theorem holds.

\begin{teo}\label{teo:reducibel}
Let $\C$ be a hyperelliptic curve of genus two with $\End(\jac{\C})\simeq\heltal{K}$, where $K$ is a quartic CM field.
Then $\jac{\C}$ is reducible if and only if $K/\QQ$ is Galois with bi-cyclic Galois group.
\end{teo}

\begin{proof}
\cite[proposition~26, p.~61]{shi}.
\end{proof}

Theorem~\ref{teo:reducibel} motivates the following definition.

\begin{defn}[Primitive, quartic CM field]\label{def:CMfieldPrimitive}
A quartic CM field~$K$ is called primitive if either $K/\QQ$ is not Galois, or $K/\QQ$ is Galois with cyclic Galois
group.
\end{defn}

\subsection{Jacobians with complex multiplication}

The CM method for constructing curves of genus~two with prescribed endomorphism ring is described in detail by
Weng~\cite{weng03}, Gaudry~\emph{et al}~\cite{gaudry} and Eisenträger and Lauter~\cite{eisen-lauter}. In short, the CM
method is based on the construction of the class polynomials of a primitive, quartic CM field~$K$ with real
subfield~$K_0$ of class number $h(K_0)=1$. The prime power $q$ has to be chosen such that $q=x\bar x$ for a number
$x\in\heltal{K}$. By \cite{weng03} we will restrict ourselves to the case $x\in\heltal{K_0}+\eta\heltal{K_0}$.

Now assume that $\jac{\C}$ has CM by a primitive, quartic CM field~$K=\QQ(\eta)$, where $\eta=i\sqrt{a+b\xi}$ and
    \begin{equation}\label{eq:xi}
    \xi=\begin{cases}
          \sqrt{D} & \textrm{if $D\not\equiv 1\pmod{4}$} \\
          \frac{1+\sqrt{D}}{2} & \textrm{if $D\equiv 1\pmod{4}$}
          \end{cases}
    \end{equation}
Here, $D$ is a square-free integer, and $K_0=\QQ(\sqrt{D})$.

\begin{defn}[$\eta$-integer]
An integer $\alpha\in\heltal{K}$ is an $\eta$-integer, if $\alpha\in\heltal{K_0}+\eta\heltal{K_0}$.
\end{defn}

If the $q$-power Frobenius endomorphism $\varphi$ under the isomorphism $\End(\jac{\C})\simeq\heltal{K}$ is given by an
$\eta$-integer $\omega$, then we can express the characteristic polynomial $P$ of $\varphi$ in terms $\omega$. Together
with Remark~\ref{rem:M} it follows from this that if $P$ splits into linear factors over $\FF_\ell$ and $D$ is not a
quadratic residue modulo $\ell$, then all $\ell$-torsion points are $\FF_{q^k}$-rational. This result is given by the
following theorem.

\begin{teo}\label{teo:Dnosquare}
Let $\C$ be a hyperelliptic curve of genus two defined over a finite field~$\FF_q$ of characteristic~$p$ and with
$\End(\jac{\C})\simeq\heltal{K}$, where $K$ is a primitive, quartic CM field with real subfield $\QQ(\sqrt{D})$. Assume
that the $q$-power Frobenius endomorphism $\varphi$ under this isomorphism is given by an $\eta$-integer~$\omega$.
Consider a prime number $\ell\neq p$ dividing $|\jac{\C}(\FF_q)|$. Assume that $\ell$ is unramified in~$\QQ(\omega)$,
and that the characteristic polynomial $\bar{P}$ of the restriction of $\varphi$ to $\jac{\C}[\ell]$ splits into linear
factors over $\FF_\ell$. Let $k$ be the multiplicative order of $q$ modulo $\ell$. If $D$ is not a quadratic residue
modulo~$\ell$, then all the $\ell$-torsion points of $\jac{\C}$ are $\FF_{q^k}$-rational.
\end{teo}

\begin{proof}
Write
    $$\omega=c_1+c_2\xi+(c_3+c_4\xi)\eta,\qquad c_i\in\ZZ.$$
Since $D$ is not a quadratic residue modulo $\ell$, it follows by lemma~\ref{lem:Dsquare} that $c_2\equiv 0\pmod{\ell}$
and $\bar{P}(X)=(X-1)^2(X-q)^2$. By theorem~\ref{teo:FrobDiagonaliserbar} it then follows that if $q\not\equiv
1\pmod{\ell}$, then the $q$-power Frobenius endomorphism is represented by the diagonal matrix $\diag(1,1,q,q)$ on
$\jac{\C}[\ell]$ with respect to an appropriate basis, whence $\jac{\C}[\ell]\subseteq\jac{\C}(\FF_{q^k})$. On the
other hand, if $q\equiv 1\pmod{\ell}$, then $\bar{P}(X)=(X-1)^4$, i.e. also in this case
$\jac{\C}[\ell]\subseteq\jac{\C}(\FF_{q^k})$.
\end{proof}

\begin{lem}\label{lem:Dsquare}
Let notation and assumptions be as in theorem~\ref{teo:Dnosquare}. Write
    $$\omega=c_1+c_2\xi+(c_3+c_4\xi)\eta,\qquad c_i\in\ZZ.$$
\begin{enumerate}
    \item If $c_2\not\equiv 0\pmod{\ell}$, then $D$ is a quadratic residue modulo $\ell$.
    \item If $c_2\equiv 0\pmod{\ell}$, then $\bar{P}(X)=(X-1)^2(X-q)^2$.
\end{enumerate}
\end{lem}

\begin{proof}
At first, assume that $D\not\equiv 1\pmod{4}$. Since the conjugates of $\omega$ are given by $\omega_1=\omega$,
$\omega_2=\bar{\omega}_1$, $\omega_3$ and $\omega_4=\bar{\omega}_3$, where
    $$\omega_3 = c_1-c_2\sqrt{D}+i(c_3-c_4\sqrt{D})\sqrt{a-b\sqrt{D}},$$
it follows that the characteristic polynomial of $\varphi$ is given by
    $$
    P(X) = \prod_{i=1}^4(X-\omega_i)
         = X^4-4c_1X^3+(2q+4(c_1^2-c_2^2D))X^2-4c_1qX+q^2.
    $$
Dividing $P(X)$ by $(X-1)(X-q)$ it then follows that $\alpha X+\beta\equiv 0\pmod{\ell}$, where
    \begin{align*}
    \beta  &\equiv q(-q^2+(4c_1-2)q+(-1+4c_2^2D-4c_1^2+4c_1)) \pmod{\ell}
    \end{align*}
Since $\beta\equiv 0\pmod{\ell}$, it follows that $4c_2^2D\equiv (2c_1-q-1)^2\pmod{\ell}$. So if $c_2\equiv
0\pmod{\ell}$, then $2c_1\equiv q+1\pmod{\ell}$, and it follows that $\bar{P}(X)=(X-1)^2(X-q)^2$.

If $D\equiv 1\pmod{4}$, then
    $$\omega_3 = c_1+c_2\frac{1-\sqrt{D}}{2}+i\left(c_3+c_4\frac{1-\sqrt{D}}{2}\right)\sqrt{a+b\frac{1-\sqrt{D}}{2}},$$
and it follows that the characteristic polynomial of $\varphi$ is given by
    $$P(X)=X^4-2cX^3+(2q+c^2-c_2^2d)X^2-2qcX+q^2,$$
where $c=2c_1+c_2$. Dividing $P(X)$ by $(X-1)(X-q)$ it then follows that $\alpha X+\beta\equiv 0\pmod{\ell}$, where
    $$
    \beta  \equiv -q(q^2+(2-2c)q+(1-2c+c^2-c_2^2D)) \pmod{\ell}.
    $$
Since $\beta\equiv 0\pmod{\ell}$, it follows that $c_2^2D\equiv (c-q-1)^2\pmod{\ell}$. As before it then follows that
if $c_2\equiv 0\pmod{\ell}$, then $\bar{P}(X)=(X-1)^2(X-q)^2$.
\end{proof}

\end{document}